\documentclass[twoside,11pt]{article}
\usepackage{graphicx,amscd,amsfonts,amsmath,amsthm,amssymb,latexsym,amsfonts,color,extsizes,multirow,algorithm,algorithmic,mathtools}
\usepackage{epsfig,float,epstopdf,array,bm,cite,cases,array,multirow,relsize,graphicx,float,booktabs}
\usepackage{mathrsfs}
\usepackage{booktabs}
\usepackage{geometry}
\usepackage{array}
\usepackage{siunitx}
\usepackage{caption}
\usepackage{booktabs}   
\usepackage[bookmarksnumbered,colorlinks,plainpages]{hyperref}
\hypersetup{citecolor=blue}
\usepackage[table]{xcolor}

\newtheorem{theorem}{\bf Theorem}

\newtheorem{example}{Example}
\newtheorem{remark}{Remark}

\begin{document}

	\title{\bf  \huge{
  A variant of the block preconditioner for  indefinite complex symmetric linear systems
}}
\author{\bf Mehdi Makhdomi$^{\dag}$ and Davod Khojasteh Salkuyeh$^{\ddag}$\thanks {\noindent Corresponding author.}~
\\	
{ $^{\dag}$ \textit{Department of Mathematics, University of Kurdistan, P.O. Box 416, Sanandaj,  Iran}} \\
{ $^{\ddag}$ \textit{Faculty of Mathematical Sciences, University of Guilan, Rasht, Iran}} \\
{ Emails: mehdi.ap.math@gmail.com,  khojasteh@guilan.ac.ir, }}

\date{}
\maketitle
\vspace{-0.5cm}
\noindent\hrulefill\\
{\bf Abstract.} 
In this paper, we propose an efficient preconditioner for solving indefinite complex symmetric linear systems within a block preconditioning framework. We analyze the convergence of the corresponding iterative method and investigate several spectral properties of the preconditioned matrix, including eigenvalue distributions and eigenvector structures.  The new preconditioner is used to accelerate the convergence of the flexible version of GMRES. Numerical experiments are presented to illustrate the effectiveness of the proposed preconditioner, and comparisons with existing block preconditioners demonstrate its superior performance.
	
	\noindent\hrulefill\\
	
		\pagestyle{myheadings}\markboth{M. Makhdomi, D. K. Salkuyeh}{ }
	\thispagestyle{empty}
\section{Introduction}\label{s.1}
\medskip
 This study is concerned with computing  an approximate solution to the system of linear equations 
 \begin{equation}\label{EQ1}
 \mathscr{A}\mathbf{x}\equiv(W + iT)\mathbf{x} = \mathbf{d}, \quad \mathbf{x} = u + iv, \quad \mathbf{d} = f + ig, 
 \end{equation}
 where \( \mathscr{A} = W + iT \in \mathbb{C}^{n \times n} \) is complex symmetric (\( W, T \in \mathbb{R}^{n \times n} \) are symmetric), \( u, v, f, g \in \mathbb{R}^n \) and \( i = \sqrt{-1} \) denotes the imaginary unit. The linear system~\eqref{EQ1} appears frequently in numerous scientific and engineering contexts, such as molecular scattering~\cite{Poirier2000}, structural dynamics~\cite{Feriani2000}. Additional examples may be found in~\cite{Benzi2008,Wu2015} and the references cited therein.
 
Assuming that the matrix $\mathscr{A}$ is large and sparse, iterative techniques are generally preferred over direct solvers, such as Gaussian elimination, for solving systems of the form \eqref{EQ1}. When the matrices $W$ and $T$ are positive semidefinite, with at least one of them being positive definite, several iterative methods and preconditioning techniques have been developed for systems such as~\eqref{EQ1}.
 
 In 2000, Axelsson and Kucherov  proposed  the C-to-R method~\cite{AxelssonKucherov2000}.
   In 2008, Bai~\cite{Bai2008} introduced several splitting for non-Hermitian linear systems. Bai et al.~\cite{BaiBenziChen2010} in 2010, proposed the modified Hermitian and skew-Hermitian splitting (MHSS) iteration method, which bypasses the need to solve a shifted skew-Hermitian linear system at each iteration. Following this, in 2011, a preconditioned variant of MHSS, known as the PMHSS method, was introduced by Bai et al.~\cite{BaiBenziChen2011}. This variant has been demonstrated to achieve remarkably superior computational efficiency compared to its predecessor.
  In 2015, Wu~\cite{Wu2015} proposed several variants of the  HSS method for solving a class of complex symmetric linear systems.
    In 2015, Salkuyeh et al. \cite{SalkuyehHezariEdalatpour2015} proposed the generalized SOR iterative method for solving a class of complex symmetric linear systems, while in the same year, Hezari et al. \cite{HezariSalkuyehEdalatpour2015} developed the preconditioned GSOR iterative method.
  To enhance the convergence rate of the GSOR method, Edalatpour et al.~\cite{Edalatpour2015} proposed an accelerated variant, termed AGSOR, which incorporates two acceleration parameters.
  In 2016, Hezari et al.~\cite{HezariSalkuyehEdalatpour2016} proposed a new iterative method for solving a class of complex symmetric linear systems. In 2017, Wang et al.~\cite{WangZhengLu2017} proposed the combination method of real and imaginary parts (CRI) for solving complex symmetric linear systems.
  In 2018, Li et al.~\cite{Li2018} constructed a symmetric block triangular splitting (SBTS) iteration method based on two distinct splittings.
   Later, in 2018, Axelsson and Salkuyeh~\cite{Axelsson2018} presented the transformed matrix preconditioner (TMP) approach.

 In order to circumvent the need for complex arithmetic, a widely adopted technique is to transform the original complex linear system~\eqref{EQ1} into an equivalent real block system of dimension $2 \times 2$. By employing the real-equivalent reformulation described in~\cite{Benzi2008Block}, we decompose the complex solution and right-hand side vectors as $\mathbf{x} = u + iv$ and $\mathbf{d} = f + ig$, respectively. Consequently, the system can be expressed in the following real block form
 
 \begin{equation}\label{orginal}
 \mathscr{A} \mathbf{x} \equiv
 \begin{bmatrix}
 T & -W \\
 W & T
 \end{bmatrix}
 \begin{bmatrix}
 u \\
 -v
 \end{bmatrix}
 =
 \begin{bmatrix}
 g \\
 f
 \end{bmatrix}
 \equiv \mathbf{d},
 \tag{2}
 \end{equation}
 From now on, we assume that $W \in \mathbb{R}^{n \times n}$ is symmetric indefinite and $T \in \mathbb{R}^{n \times n}$ is symmetric positive definite.
 
 Due to the high dimensionality and sparse nature of the coefficient matrix $\mathscr{A}$ in the $2 \times 2$ block linear system~\eqref{orginal}, applying direct factorization methods typically incurs severe fill-in. This not only imposes prohibitive memory requirements but also significantly degrades computational efficiency, rendering direct solvers impractical for large-scale, real-world applications. To address these limitations, extensive research has focused on developing iterative schemes that preserve and exploit the system's sparsity. Among these, Krylov subspace methods—most notably the generalized minimal residual (GMRES) algorithm~\cite{Saadgmres} and its flexible variant (FGMRES)~\cite{Saadfgmres}—have proven highly effective, as they rely solely on matrix-vector multiplications, thereby preserving the original sparsity pattern.
 
 In the following, we review some of the existing research in this area, particularly for the case where the matrix \(W \) is symmetric indefinite and \(T\) is SPD, 
 	for solving complex symmetric linear systems~\eqref{orginal}, numerous iterative methods have been proposed in the literature, encompassing both stationary iteration schemes and Krylov subspace methods, the latter often being equipped with preconditioners to enhance convergence.
  In 2017, Li and Wu~\cite{WuLi2017} introduced the modified positive/negative stable splitting (MPNS) method. In 2021, Pourbagher and Salkuyeh~\cite{PourbagherSalkuyeh2021} put forward the symmetric positive definite and negative stable splitting (SNSS) method. In 2022, Axelsson et al.~\cite{AxelssonPourbagherSalkuyeh2022} developed three iteration schemes for system~\eqref{EQ1} under the assumption that \(W = W_1 - W_2\), with both \(W_1\) and \(W_2\) being SPD. They also employed the resulting preconditioners to accelerate the convergence of GMRES when applied to~\eqref{EQ1}. Also in 2022, Zheng et al.~\cite{ZhengZengZhang2022} proposed the VPMHSS method.
 		
 		More recently, in 2025, Chen and Wu~\cite{ChenWu2025} proposed two efficient iteration methods for solving complex symmetric indefinite linear systems. In the same year, Liang and Dou~\cite{LiangDou2025} proposed modified CRI (combination of real and imaginary parts) iteration methods. Also in 2025, Salkuyeh~\cite{Salkuyeh2025} proposed a preconditioner for complex symmetric systems of linear equations with an indefinite Hermitian part.
 	
Given that in this study, the matrices $T$ and $W$ are symmetric positive definite and symmetric indefinite, respectively, we introduce some effective methods that share this structure in detail.

It is well established that the convergence and performance of Krylov subspace methods depend heavily on the quality of the preconditioner. Consequently, a substantial body of literature has been dedicated to designing and analyzing efficient preconditioning techniques. Specifically, building upon the HSS iteration framework introduced by Bai et al.~\cite{Bai2003}, various adapted HSS-based iteration schemes have been developed for the real $2 \times 2$ block system~\eqref{orginal}, formulated as follows
  \begin{equation}
  \mathscr{P}_{\text{HSS}} = \frac{1}{2\alpha}
  \begin{bmatrix}
  \alpha I + T & 0 \\
  0 & \alpha I + T
  \end{bmatrix}
  \begin{bmatrix}
  \alpha I - W & W \\
  W & \alpha I
  \end{bmatrix}
  = \frac{1}{2}
  \begin{bmatrix}
  \alpha I + T & \left(I + \frac{1}{\alpha}T\right)W \\[6pt]
  \left(I + \frac{1}{\alpha}T\right)W & \alpha I + T
  \end{bmatrix}.
  \label{eq:HSS_prec}
  \end{equation}
  
  Subsequently, Zhang and Dai~\cite{ZhangDai2017} employed a relaxation strategy to construct the block splitting (BS) preconditioner, formulated as
  \begin{equation}\label{BS}
  \mathscr{P}_{\text{BS}} =
  \begin{bmatrix}
  I & -W \\
  \frac{1}{\alpha}W & T
  \end{bmatrix}
  \begin{bmatrix}
  \alpha I + T & 0 \\
  0 & I
  \end{bmatrix}
  =
  \begin{bmatrix}
  \alpha I + T & -W \\[4pt]
  W\left(I + \frac{1}{\alpha}T\right) & T
  \end{bmatrix}.
  \end{equation}
  
  Building upon the HSS preconditioning framework and incorporating relaxation techniques, Shen and Shi~\cite{ShenShi2018} introduced a variant of the HSS preconditioner (VHSS), defined by
  \begin{equation}\label{VHSS}
  \mathscr{P}_{\text{VHSS}} = \frac{1}{2\alpha}
  \begin{bmatrix}
  \alpha I + T & 0 \\
  0 & 2\alpha I
  \end{bmatrix}
  \begin{bmatrix}
  \alpha I - W & W \\
  W & T
  \end{bmatrix}
  = \frac{1}{2}
  \begin{bmatrix}
  \alpha I + T & \left(I + \frac{1}{\alpha}T\right)W \\[6pt]
  2W & 2T
  \end{bmatrix}.
  \end{equation}
  
  More recently, Balani and Hajarian~\cite{BalaniHajarian2023} proposed a modified block product (MBP) preconditioner of the form
  \begin{equation}\label{MBP}
  \mathscr{P}_{\text{MBP}} =
  \begin{bmatrix}
  T & 0 \\
  W & \alpha T
  \end{bmatrix}
  \begin{bmatrix}
  I & -\frac{1}{\alpha}W \\
  0 & I
  \end{bmatrix}
  =
  \begin{bmatrix}
  T & -\frac{1}{\alpha}TW \\[4pt]
  W & \alpha T - \frac{1}{\alpha}W^2
  \end{bmatrix}.
  \end{equation}
  
  In this paper, we propose a variant of the block preconditioner for complex symmetric systems, which we refer to as the VBP preconditioner. This new preconditioner is derived from a matrix splitting approach and is designed for solving the large sparse $2 \times 2$ block complex symmetric indefinite linear system of the form \eqref{orginal}.
  
  Throughout this paper, $ \rho(A) $, $ \textrm{tr}(A) $, $ \Vert A\Vert_{2} $ and $ \Vert A\Vert _{F} $ denote the spectral radius, trace, the Euclidean norm and the Frobenius norm of matrix $ A $, respectively.  $ \lambda_{\max}(A) $ and $  \lambda_{\min}(A) $  denote the largest eigenvalue and the smallest eigenvalue of matrix $ A  $, respectively.
 For a vector $ x \in \mathbb{C}^{n} $, $ x^{\ast} $ is used for the conjugate transpose of $ x $. For two vectors $ x $ and $ y $, the 
 \textsc{Matlab} notation $ (x;y) $ is used for $(x^{T}, y^{T})^{T}$ (or for $(x^{\ast},
 y^{\ast})^{\ast})$. 
 
 The remainder of this paper is organized as follows. In Section~\ref{Sec2}, we present the detailed construction of the proposed VBP preconditioner and analyze the convergence properties of the associated iterative method. Practical implementation aspects, including computational complexity and storage requirements, are also discussed in this section. Section~\ref{Sec3} is devoted to a thorough investigation of the spectral properties of the preconditioned matrix $\mathscr{P}_{\text{VBP}}^{-1}\mathcal{A}$, providing theoretical insights into the convergence behavior. The strategy for selecting the optimal preconditioner parameters $\alpha$ and $ \gamma $ is described in Section~\ref{Sec4}. In Section~\ref{Sec5}, we conduct extensive numerical experiments to evaluate the performance, efficiency, and robustness of the proposed approach in comparison with existing preconditioners such as VHSS, BS, and MBP. Finally, Section~\ref{Sec6} concludes the paper with some concluding remarks and discusses possible directions for future research.
 

\section{The VBP preconditioner and convergence analysis}\label{Sec2}
In this section, we introduce a variant of the block preconditioner (denoted as VBP) for complex symmetric systems. Let $\alpha > 0$ and $ \gamma $ be a positive real parameters. Using the $2 \times 2$ block matrix $\mathscr{A}$ given in~\eqref{orginal}, we construct the preconditioner block approach as follows
\begin{equation}\label{precondi}
\mathscr{P}_{\text{VBP}} = 
\begin{bmatrix} 
T+\alpha I & -\gamma W \\ 
0 & I 
\end{bmatrix} 
\begin{bmatrix} 
I & 0 \\ 
W & T 
\end{bmatrix} 
= 
\begin{bmatrix} 
T+\alpha I-\gamma W^2 & -\gamma WT \\ 
W & T 
\end{bmatrix}.
\end{equation}

This preconditioner, $\mathscr{P}_{\text{VBP}}$, may alternatively be formulated via the splitting approach
\begin{equation}\label{EQA}
\mathscr{A} = \mathscr{P}_{\text{VBP}} - \mathscr{R}_{\text{VBP}},
\end{equation}
where the remainder matrix is given by
\begin{equation}\label{w1}
\mathscr{R}_{\text{VBP}} = \mathscr{P}_{\text{MBP}} - \mathscr{A} = 
\begin{bmatrix} 
\alpha I-\gamma W^2 & W -\gamma WT \\ 
0 & 0 
\end{bmatrix}.
\end{equation}

The selection strategy for these parameters will be discussed in Section~\ref{Sec4}.

Utilizing the splitting established in~\eqref{EQA} for our coefficient matrix $\mathscr{A}$, we can introduce the subsequent VBP iterative scheme to solve the $2 \times 2$ block linear system~\eqref{orginal}

\vspace{0.5cm}
\noindent \textbf{Method 2.1 (VBP Iteration Scheme):} Given an initial vector $(u^{(0)}; v^{(0)}) \in \mathbb{R}^{2n}$, the proposed VBP method produces a sequence of approximate solutions $\{(u^{(k)}; v^{(k)})\}$ utilizing the iteration
\begin{equation}\label{iter}
\mathscr{P}_{\text{VBP}} 
\begin{bmatrix} 
u^{(k+1)} \\ 
v^{(k+1)} 
\end{bmatrix} 
= 
\mathscr{R}_{\text{VBP}}
\begin{bmatrix} 
u^{(k)} \\ 
v^{(k)} 
\end{bmatrix} 
+ 
\begin{bmatrix} 
g \\ 
f 
\end{bmatrix}.
\end{equation}

We can rewrite the iteration presented in~\eqref{iter} into a more compact form
\begin{equation}\label{eqs}
\mathbf{u}^{(k+1)} = \mathscr{L}_{\text{VBP}} \mathbf{u}^{(k)} + c,
\end{equation}
in which the iteration matrix is defined as
\begin{equation}
\mathscr{L}_{\text{VBP}} = \mathscr{P}_{\text{VBP}}^{-1} \mathscr{R}_{\text{VBP}} = 
\begin{bmatrix} 
T+\alpha I-\gamma W^2 & -\gamma WT \\ 
W & T 
\end{bmatrix}^{-1} 
\begin{bmatrix} 
\alpha I-\gamma W^2 & W -\gamma WT \\ 
0 & 0 
\end{bmatrix},
\end{equation}
and the constant vector is $c = \mathscr{P}_{\text{VBP}}^{-1} \mathbf{d}$.

To wrap up this section, we examine how the $\mathscr{P}_{\mathrm{VPT}}$ preconditioner is implemented within the framework of Krylov subspace techniques to solve the linear system $\mathscr{A}\mathbf{x} = \mathbf{d}$. 
At each iteration of a Krylov method such as GMRES or its flexible variant FGMRES, it is necessary to evaluate the inverse action of the preconditioning matrix. This translates to finding the solution for a linear equation structured as
$$ \mathscr{P}_{\text{VBP}}z = r, $$ 
in which the vectors are partitioned as $z = (z_{1}; z_{2} )$ and $r = (r_{1}; r_{2})$. Here, the respective block components belong to the vector spaces $z_{1}, r_{1} \in \mathbb{R}^{n}$, $z_{2}, r_{2} \in \mathbb{R}^{n}$.
The sequential procedure for determining $z$ such that $\mathscr{P}_{\text{VBP}}z = r$ is outlined below.

\medskip
\noindent\textbf{Algorithm 1.} \textit{Evaluating $z = \mathscr{P}_{\mathrm{VBP}}^{-1}r$ }\\[2mm]
1. Solve $(\alpha I+T) z_{1} = r_{1} + \gamma Wr_2$ to find $z_{1}$; \\ [1.5mm]
2. Solve $T z_{2} = r_2-Wz_1$ to obtain $z_{2}$;\\[1.5mm]
3. Form the vector $z = (z_1; z_2)$.
\medskip

As demonstrated in Algorithm~1, determining the solution requires handling two internal linear subsystems at the first and second steps. These systems feature the following coefficient matrices
\[
 \alpha I+T
\quad\text{and}\quad
T.
\]
Because both of these matrices possess the symmetric positive definite (SPD) property, one can compute their solutions directly via Cholesky decomposition or iteratively via the Conjugate Gradient (CG) algorithm. 
It is worth noting that if an approximate (inexact) inner solver is employed for these subsystems, adopting the FGMRES method becomes essential to guarantee that the outer iterations remain robust and successfully converge.

Next, we explore the convergence characteristics associated with the VBP iteration scheme applied to the block linear system~\eqref{orginal}. Additionally, we discuss the spectral behavior of the preconditioned coefficient matrix $\mathscr{P}_{\text{VBP}}^{-1}\mathscr{A}$. 

As established in the literature~\cite{Saad2003}, the sequence of iterative vectors $\mathbf{u}^{(k+1)}$ generated by~\eqref{eqs} will successfully converge to the exact solution of~\eqref{orginal} strictly when the spectral radius of the iteration matrix $\mathscr{L}_{\text{VBP}}$, denoted by $\rho(\mathscr{L}_{\text{VBP}})$, is less than unity. 

\begin{theorem}\label{th1}
	Let \( W \in \mathbb{R}^{n \times n} \) be a symmetric indefinite matrix and let \( T \in \mathbb{R}^{n \times n} \) be symmetric positive definite. Let \( \alpha > 0 \) and \( \gamma > 0 \). If the parameter \( \alpha \) satisfies the condition
\begin{equation}\label{cond1}
	\alpha > \frac{b-c}{2},
\end{equation}
	where
	\begin{equation}\label{amu}
	c = x^{\ast} T x>0, \qquad b = x^{\ast} W T^{-1}W  x\ge 0,
\end{equation}
	with \( x \in \mathbb{C}^n \) such that \( \| x \|_2 = 1 \), then the iterative scheme defined in~\eqref{iter} converges to the unique solution for any initial guess \( x^{(0)} \in \mathbb{R}^n \).
\end{theorem}

\begin{proof}
	Let    $ (\lambda, \textbf{u}=(x;y)) $  be an eigenpair of  the  iteration matrix $ \mathscr{L}_{\text{VBP}} $, such that $  \| x\|_2=1 $. Then, we get
	$\mathscr{L}_{\text{VBP}}(x;y)=\lambda (x;y),$
	which is equivalent to
	\begin{equation}\label{eqq}
	\mathscr{R}_{\text{VBP}} \begin{bmatrix} x \\ y \end{bmatrix}=\lambda \mathscr{P}_{\text{VBP}} \begin{bmatrix} x \\ y \end{bmatrix}.
\end{equation}
	It is straightforward to show that this equation is equivalent to 	
	\begin{align}\label{eq.14}
	\begin{cases}
\alpha x-\gamma W^2x+Wy-\gamma WTy=\lambda (Tx +\alpha x-\gamma W^2x-\gamma WTy ),\\
	0=\lambda (Wx+Ty).
	\end{cases}
	\end{align} 
If $\lambda=0$, then there is nothing to prove. So we assume that $  \lambda\neq 0 $.
	
We now claim that \( x\neq 0 \). Indeed, if \( x = 0 \), then the second equation in~\eqref{eq.14} reduces to \( Ty = 0 \). Since \( T \) is positive definite, it is nonsingular, and therefore \( y = 0 \). Consequently, \( \mathbf{u} = 0 \), which contradicts the assumption that \( \mathbf{u} \) is an eigenvector.
	
	 So we assume that $\lambda \neq 0$ and $ x\neq 0 $.  In this case, from the second relation of Eq. \eqref{eq.14},  we get 
	$$y=-T^{-1}Wx.$$  
	Substituting $y$ in the first equation of \eqref{eq.14} yields
	\begin{equation}\label{eq.15}
	\alpha x-WT^{-1}Wx=\lambda (Tx+\alpha x).
	\end{equation}
	
	Now, multiplying both sides of relation~\eqref{eq.15} from the left by the vector $ 	x^{\ast} $
	, we obtain
	\begin{equation}\label{eq.171}
	\alpha x^{\ast} x-x^{\ast}WT^{-1}Wx=\lambda (x^{\ast}Tx+\alpha x^{\ast} x).
	\end{equation}
	Equation~\eqref{eq.171} yields
	
	\begin{equation}\label{A8}
	\lambda=\frac{\alpha -b}{\alpha+ c}.
	\end{equation}
	where  $ c = x^{\ast} Tx$ and $b = x^{\ast} W T^{-1}Wx $. 
	For the convergence condition to hold, we must have
$|\lambda| < 1.$
	Hence, from~\eqref{A8}, we conclude tha
	\begin{equation*}\label{A21}
	\alpha > \frac{b-c}{2}. 
	\end{equation*}
	Thus, the proof is complete.
\end{proof}
\begin{remark}
	Since the values of $c$ and $b$ defined in Theorem~\ref{th1} are  positive and non-negative, respectively, it follows that all eigenvalues of the iteration matrix are real. Therefore, under condition~\eqref{cond1}, we have $-1 < \lambda < 1$. Furthermore, if
	\begin{equation}
	\alpha > \frac{1}{2} \left( \lambda_{\max}(WT^{-1}W) - \lambda_{\min}(T) \right),
	\end{equation}
	then condition~\eqref{cond1} holds and the method converges. Moreover, we have
	\begin{align*}
	x^{\ast} W T^{-1} W x &= \frac{(Wx)^{\ast} T^{-1} (Wx)}{(Wx)^{\ast} (Wx)} \cdot (x^{\ast} W^2 x) \\
	&\le \lambda_{\max}(T^{-1}) \cdot \lambda_{\max}(W^2) \\
	&= \frac{\sigma_{\max}^2(W)}{\lambda_{\min}(T)}.
	\end{align*}
	Therefore, by Eq.~\eqref{cond1}, if $\alpha > \max\{0, \beta\}$ with
	\begin{equation}
	\beta = \frac{1}{2} \left( \frac{\sigma_{\max}^2(W)}{\lambda_{\min}(T)} - \lambda_{\min}(T) \right),
	\end{equation}
	the convergence of the method is guaranteed.
\end{remark}
It should be remarked that the parameter \( \gamma \) has no influence on the convergence analysis of the proposed method, nor on the spectral properties of the preconditioned matrix. Nevertheless, this parameter proves to be effective in numerical experiments.

\section{Eigenvalue analysis of $ \mathscr{P}_{\text {VBP}}^{-1}\mathscr{A}$ }\label{Sec3}
\medskip

\begin{theorem}\label{th3}
Under the assumptions of Theorem~1, the eigenvalues of the preconditioned matrix 
$\mathscr{P}_{\mathrm{VBP}}^{-1} \mathscr{A}$ 
are real and positive, and lie in the interval
	\begin{equation}\label{spctt}
	\sigma(\mathscr{P}_{\mathrm{VBP}}^{-1} \mathscr{A}) 
	\subset 
	\left[
	\frac{\lambda_{\min}(T) + \sigma_{\min}^2(W)/\lambda_{\max}(T)}
	{\alpha + \lambda_{\max}(T)},
	\,
	\frac{\lambda_{\max}(T) + \sigma_{\max}^2(W)/\lambda_{\min}(T)}
	{\alpha + \lambda_{\min}(T)}
	\right].
	\end{equation}
	where $\lambda_{\min}(T)$ and $\lambda_{\max}(T)$ are the smallest and largest eigenvalues of $T$, respectively, and $\sigma_{\min}(W)$ and $\sigma_{\max}(W)$ are the smallest and largest singular value of $W$, respectively.
\end{theorem}

\begin{proof}
	Let $\zeta$ be an eigenvalue of the preconditioned matrix $\mathscr{P}_{\mathrm{VBP}}^{-1} \mathscr{A}$. Then, it can be expressed as $\zeta = 1 - \lambda$, where $\lambda$ is an eigenvalue of the iteration matrix $\mathscr{L}_{\mathrm{VBP}}$. Using Eq.~\eqref{A8}, we can write
	\begin{equation}\label{eq:zetamu}
	\zeta = 1 - \frac{\alpha - b}{\alpha + c} = \frac{c + b}{c+\alpha },
	\end{equation}
	where $c = x^{\ast} T x$ and $b = x^{\ast} W T^{-1} W x$.
	
	Given that $T$ is symmetric positive definite, the definitions of the quadratic forms $c$ and $b$ ensure that $\zeta$ is  real and positive.
	
	Under the assumptions stated in Theorem~\ref{th1}, the variables $c$ and $b$ are bounded as follows
	\begin{equation}\label{eq:bounds_cb}
	\lambda_{\min}(T) \le c \le \lambda_{\max}(T), \quad \text{and} \quad
	\frac{\sigma_{\min}^2(W)}{\lambda_{\max}(T)} \le b \le \frac{\sigma_{\max}^2(W)}{\lambda_{\min}(T)}.
	\end{equation}
	
	To determine the bounds for $\zeta$, we apply the inequalities from \eqref{eq:bounds_cb} to the expression in \eqref{eq:zetamu}. Assuming $\alpha > 0$ and knowing $c > 0$, we can establish the lower bound by minimizing the numerator and maximizing the denominator
	\[
	\zeta \ge \frac{\lambda_{\min}(T) + \frac{\sigma_{\min}^2(W)}{\lambda_{\max}(T)}}{\alpha + \lambda_{\max}(T)}.
	\]
	Similarly, we obtain the upper bound by maximizing the numerator and minimizing the denominator
	\[
	\zeta \le \frac{\lambda_{\max}(T) + \frac{\sigma_{\max}^2(W)}{\lambda_{\min}(T)}}{\alpha + \lambda_{\min}(T)}.
	\]
	
	Consequently, the spectrum of the preconditioned matrix, denoted by $\sigma(\mathscr{P}_{\mathrm{VBP}}^{-1} \mathscr{A})$, is contained within the following interval
	\begin{equation}\label{spct}
	\sigma(\mathscr{P}_{\mathrm{VBP}}^{-1} \mathscr{A}) 
	\subset 
	\left[
	\frac{\lambda_{\min}(T) + \sigma_{\min}^2(W)/\lambda_{\max}(T)}
	{\alpha + \lambda_{\max}(T)},
	\,
	\frac{\lambda_{\max}(T) + \sigma_{\max}^2(W)/\lambda_{\min}(T)}
	{\alpha + \lambda_{\min}(T)}
	\right].
	\end{equation}
	This completes the proof.
\end{proof}
\begin{theorem}
Assuming the conditions stated in Theorem~\ref{th1} are satisfied, the preconditioned matrix $\mathscr{P}_{\mathrm{VBP}}^{-1}\mathscr{A}$ possesses the eigenvalue $1$ with algebraic multiplicity at least $n$. Furthermore, the remaining eigenvalues are entirely real and reside in a positive real interval~\eqref{spct}.
\end{theorem}
\begin{proof}
	A straightforward algebraic manipulation reveals that the preconditioned matrix can be expressed as
	\begin{equation}\label{Theta}
	\mathscr{A}\mathscr{P}_{\mathrm{VBP}}^{-1} = 
	\begin{bmatrix}
	\Theta_1 & \Theta_2 \\
	0 & I
	\end{bmatrix},
	\end{equation}
	where
	\begin{equation}
\Theta_1=(T + W T^{-1} W)(T + \alpha I)^{-1},\quad	\Theta_2 = \gamma (T + W T^{-1} W)(T + \alpha I)^{-1} W-WT^{-1}.
	\end{equation}
	Note that the matrix $\mathscr{A}\mathscr{P}_{\mathrm{VBP}}^{-1}$ is similar to $\mathscr{P}_{\mathrm{VBP}}^{-1}\mathscr{A}$; hence, they have identical spectra.
	
	From this block triangular structure, it is immediately clear that the matrix $\mathscr{P}_{\mathrm{VBP}}^{-1}\mathscr{A}$ admits the eigenvalue $1$ with algebraic multiplicity at least $n$. Moreover, the remaining eigenvalues coincide with those of the matrix $\Theta_1$, which have been characterized in relation~\eqref{eq:zetamu}. As established therein, these eigenvalues are all real and positive, and they lie within the interval specified in~\eqref{spct}.
\end{proof}
\begin{theorem}
	Assume that the preconditioner $\mathscr{P}_{\mathrm{VBP}}$ is given by \eqref{precondi}. Then, the linear independence of $n + r$ eigenvectors of the preconditioned matrix $\mathscr{P}_{\mathrm{VBP}}^{-1} \mathscr{A}$ (with $0 \leq r \leq n$) admits the following characterization
	
	\medskip
	\noindent\textbf{Case I ($\zeta = 1$).}  
	The eigenvalue $\zeta = 1$ has eigenvectors of the form
	\[
	\mathbf{u} = \begin{bmatrix} K y \\ y \end{bmatrix}, \qquad y \neq 0,
	\]
	where 
	\[
	K = (\alpha I - \gamma W^2)^{-1}(\gamma W T - W)
	\]
	and $\frac{\alpha}{\gamma} \neq {\mu}^2$ for every eigenvalue $\mu$ of $W$.
	
	\medskip
	\noindent\textbf{Case II ($\zeta \neq 1$).}  
	For each $\zeta \neq 1$, a corresponding eigenvector is given by
	\[
	\mathbf{u} = \begin{bmatrix} x \\ -T^{-1} W x \end{bmatrix},
	\qquad x \neq 0.
	\]
\end{theorem}

\begin{proof}
	Let $\mathbf{u} = (x;\, y)$ be an eigenvector associated with the eigenvalue $\zeta$, i.e.,
	\[
	\mathscr{P}_{\mathrm{VBP}}^{-1} \mathscr{A} \mathbf{u}
	= \zeta \mathbf{u},
	\qquad\text{equivalently}\qquad
	\mathscr{A} \mathbf{u}
	= \zeta \, \mathscr{P}_{\mathrm{VBP}} \mathbf{u}.
	\]
	This relation expands to
	\begin{equation}\label{eqss}
	\begin{bmatrix}
	T & -W \\
	W & T
	\end{bmatrix}
	\begin{bmatrix}
	x \\ y 
	\end{bmatrix}
	=
	\zeta
	\begin{bmatrix}
	T + \alpha I - \gamma W^2 & -\gamma W T \\
	W & T
	\end{bmatrix}
	\begin{bmatrix}
	x \\ y 
	\end{bmatrix},
	\end{equation}
	where $x, y \in \mathbb{R}^n$.
	Equation~\eqref{eqss} is equivalent to the system
	\begin{align}
	T x - W y &= \zeta \bigl( T x + \alpha x - \gamma W^2 x - \gamma W T y \bigr), \label{eqd3} \\
	W x + T y &= \zeta ( W x + T y ). \label{eqx1}
	\end{align}
	The eigenvalue $\zeta$ cannot be zero because both matrices $\mathscr{P}_{\mathrm{VBP}}$ and $\mathscr{A}$ are nonsingular. Hence, we consider the following cases.
	
	\medskip
	\noindent\textbf{Case 1 ($\zeta = 1$):} 
	Substituting $\zeta = 1$ into equation~\eqref{eqd3} yields
	\begin{equation}\label{A1A}
	(\gamma W T - W) y = (\alpha I - \gamma W^2) x.
	\end{equation}
	Let $\mu$ be an eigenvalue of the matrix $W$. Then the matrix
	\[
	\alpha I - \gamma W^2
	\]
	is nonsingular if and only if
	\[
	\frac{\alpha}{\gamma} \neq \mu^2
	\]
	for every eigenvalue $\mu$ of $W$.
	
	Therefore, from relation~\eqref{A1A}, we conclude that
	\[
	x = (\alpha I - \gamma W^2)^{-1} (\gamma W T - W) y.
	\]
	Hence, the eigenvector corresponding to the eigenvalue $\zeta = 1$ is given by
	\[
	\mathbf{u} = \begin{bmatrix}
	(\alpha I - \gamma W^2)^{-1} (\gamma W T - W) y \\
	y
	\end{bmatrix},
	\quad y \neq 0.
	\]
	
	\medskip
	\noindent\textbf{Case 2 ($\zeta \neq 1$):} 
	For $\zeta \neq 1$, equation~\eqref{eqx1} gives $y = -T^{-1} W x$. Substituting this into equation~\eqref{eqd3}, we have
	\begin{equation}\label{A14}
	\bigl( (1 - \zeta) T + WT^{-1}W - \alpha \zeta  \bigr) x = 0.
	\end{equation}
	
	Since $\mathscr{P}_{\mathrm{VBP}}^{-1} \mathscr{A}$ and $\mathscr{L}_{\mathrm{VBP}}$ share the same eigenvectors, and by Theorem~\ref{th1}, it follows that $x \neq 0$.  
	Without loss of generality, we may assume $\|x\|_2 = 1$.  
	Premultiplying~\eqref{A14} by $x^*$ yields the quadratic equation
	\begin{equation}\label{tqad}
	\zeta = \frac{c+b }{c + \alpha },
	\end{equation}
	where $c$ and $b$ are as defined in~\eqref{amu}.  
	
	Thus, $\zeta$ in~\eqref{tqad} is an eigenvalue of 
	$\mathscr{P}_{\mathrm{VBP}}^{-1} \mathscr{A}$, and the corresponding eigenvector is
	\[
	\mathbf{u} = \begin{bmatrix}
	x \\
	-T^{-1} W x
	\end{bmatrix},
	\quad x \neq 0.
	\]
	
	\medskip
	In what follows, we prove that the set of eigenvectors corresponding to $\zeta = 1$ and $\zeta \neq 1$ are linearly independent.  
	First, we determine the number of eigenvectors
	
	\begin{itemize}
		\item For $\zeta = 1$, since $y \in \mathbb{R}^n$ is arbitrary, there are exactly $n$ linearly independent vectors, which can be written as
		\[
		\mathbf{u}_j^{(1)} = \begin{bmatrix} K e_j \\ e_j \end{bmatrix}, \qquad j = 1, \dots, n,
		\]
		where $\{e_j\}_{j=1}^n$ is the standard basis of $\mathbb{R}^n$.
		
		\item For $\zeta \neq 1$, assume there exist $r$ linearly independent vectors
		\[
		\mathbf{u}_j^{(2)} = \begin{bmatrix} x_j \\ -T^{-1} W x_j \end{bmatrix}, \qquad j = 1, \dots, r,
		\]
		where $0 \leq r \leq n$ and $\{x_j\}_{j=1}^r$ is a linearly independent set in $\mathbb{R}^n$.
	\end{itemize}
	
	Now suppose the following linear combination equals zero
	\begin{equation}\label{u1}
	\sum_{j=1}^n c_j^{(1)} \mathbf{u}_j^{(1)}
	+
	\sum_{j=1}^r c_j^{(2)} \mathbf{u}_j^{(2)}
	=
	\mathbf{0}.
	\end{equation}
	
	Substituting the vector forms, we obtain
	\begin{equation}
	\sum_{j=1}^n c_j^{(1)} 
	\begin{bmatrix} K e_j \\ e_j \end{bmatrix}
	+
	\sum_{j=1}^r c_j^{(2)} 
	\begin{bmatrix} x_j \\ -T^{-1} W x_j \end{bmatrix}
	=
	\begin{bmatrix} 0 \\ 0 \end{bmatrix}.
	\end{equation}
	
	Since each $\mathbf{u}_j^{(1)}$ is an eigenvector corresponding to the eigenvalue $1$, we have
	\[
	\mathscr{P}_{\mathrm{VBP}}^{-1} \mathscr{A} \, \mathbf{u}_j^{(1)}
	=
	1 \cdot \mathbf{u}_j^{(1)}.
	\]
	Similarly, for each $\mathbf{u}_j^{(2)}$ corresponding to the eigenvalue $\zeta_j$, we have
	\[
	\mathscr{P}_{\mathrm{VBP}}^{-1} \mathscr{A} \, \mathbf{u}_j^{(2)}
	=
	\zeta_j \cdot \mathbf{u}_j^{(2)}.
	\]
	
	Now multiply both sides of equation~\eqref{u1} on the left by the matrix 
	$\mathscr{P}_{\mathrm{VBP}}^{-1} \mathscr{A}$. This yields
	\begin{equation}\label{u2}
	\sum_{j=1}^n c_j^{(1)} \mathbf{u}_j^{(1)}
	+
	\sum_{j=1}^r \zeta_j c_j^{(2)} \mathbf{u}_j^{(2)}
	=
	\mathbf{0}.
	\end{equation}
	
	Subtracting equation~\eqref{u1} from equation~\eqref{u2}, the terms corresponding to the first group (with eigenvalue $1$) cancel out, since they are identical in both equations. For the second group, we obtain
	\[
	\zeta_j c_j^{(2)} - c_j^{(2)} = (\zeta_j - 1) c_j^{(2)}.
	\]
	Thus
	\begin{equation}\label{u4}
	\sum_{j=1}^r (\zeta_j - 1) c_j^{(2)} \mathbf{u}_j^{(2)}
	=
	\mathbf{0}.
	\end{equation}
	
	Equation~\eqref{u4} is a linear combination of the vectors 
	$\mathbf{u}_j^{(2)}$ equal to zero. However, by assumption, these $r$ vectors are linearly independent. Therefore, the only possibility is that all coefficients are zero
	\[
	(\zeta_j - 1) c_j^{(2)} = 0, \qquad \forall j = 1, \dots, r.
	\]
	Since $\zeta_j \neq 1$, we have $\zeta_j - 1 \neq 0$, and consequently
	\begin{equation}\label{u5}
	c_j^{(2)} = 0, \qquad \forall j = 1, \dots, r.
	\end{equation}
	
	Now substituting~\eqref{u5} into equation~\eqref{u1}, we get
	\begin{equation}\label{u6}
	\sum_{j=1}^n c_j^{(1)} \mathbf{u}_j^{(1)}
	=
	\mathbf{0},
	\end{equation}
	which is equivalent to
	\begin{equation}\label{u7}
	\sum_{j=1}^n c_j^{(1)} 
	\begin{bmatrix} K e_j \\ e_j \end{bmatrix}
	=
	\begin{bmatrix} 0 \\ 0 \end{bmatrix}.
	\end{equation}
	
	From the second component of equation~\eqref{u7}, we have
	\[
	\sum_{j=1}^n c_j^{(1)} e_j = 0.
	\]
	Since $\{e_j\}_{j=1}^n$ is the standard basis of $\mathbb{R}^n$, they are linearly independent, and therefore
	\begin{equation}\label{u8}
	c_j^{(1)} = 0, \qquad \forall j = 1, \dots, n.
	\end{equation}
	
	From~\eqref{u5} and \eqref{u8}, it follows that all coefficients are zero. Hence, only the trivial linear combination can yield the zero vector. Consequently, the set
	\[
	\left\{ \mathbf{u}_j^{(1)} \right\}_{j=1}^n \cup 
	\left\{ \mathbf{u}_j^{(2)} \right\}_{j=1}^r
	\]
	consists of $n + r$ linearly independent vectors.
\end{proof}

\begin{theorem}\label{thm:minimal_poly}
	Let the VBP preconditioner $\mathscr{P}_{\mathrm{VBP}}$ be defined as in \eqref{precondi}. Then the degree of the minimal polynomial of the preconditioned matrix $\mathscr{P}_{\mathrm{VBP}}^{-1} \mathscr{A}$ is at most $n + 1$.
\end{theorem}

\begin{proof}
	From \eqref{Theta}, we recall that the preconditioned matrix admits the following block upper triangular representation
	\begin{equation}
	\mathscr{P}_{\mathrm{VBP}}^{-1} \mathscr{A}
	= 
	\begin{bmatrix}
	\Theta_1 & \Theta_2 \\
	0 & I 
	\end{bmatrix},
	\label{eq:precond_block}
	\end{equation}
	where $I \in \mathbb{R}^{n \times n}$ is the identity matrix, $\Theta_1 \in \mathbb{R}^{n \times n}$, and $\Theta_2 \in \mathbb{R}^{n \times n}$ are given by
	\begin{align}
	\Theta_1 &= (T + WT^{-1}W)(T+\alpha I)^{-1}, \nonumber \\
	\Theta_2 &=\gamma (T + W T^{-1} W)(T + \alpha I)^{-1} W-WT^{-1}.
	\label{eq:theta_def}
	\end{align}
	
	Let $\lambda_1, \lambda_2, \ldots, \lambda_n$ denote the eigenvalues of the matrix $\Theta_1 \in \mathbb{R}^{n \times n}$. Since the matrix in \eqref{eq:precond_block} is block upper triangular, its spectrum is the union of the spectra of its diagonal blocks. Consequently, the eigenvalues of $\mathscr{P}_{\mathrm{VBP}}^{-1} \mathscr{A}$ consist of $\lambda_1, \ldots, \lambda_n$ (from the $(1,1)$ block $\Theta_1$) and $1$ with multiplicity $n$ (from the $(2,2)$ block $I$). Hence, the characteristic polynomial of $\mathscr{P}_{\mathrm{VBP}}^{-1} \mathscr{A}$ is given by
	\begin{equation}
	\Phi_{\mathscr{P}_{\mathrm{VBP}}^{-1} \mathscr{A}}(\lambda)
	= \det\left( \mathscr{P}_{\mathrm{VBP}}^{-1} \mathscr{A} - \lambda I \right)
	= \prod_{i=1}^n (\lambda - \lambda_i) (\lambda - 1)^n.
	\label{eq:char_poly}
	\end{equation}
	
	Now define the polynomial $\Psi(\lambda)$ of degree $n + 1$ as follows
	\begin{equation}
	\Psi(\lambda) := \left( \prod_{i=1}^n (\lambda - \lambda_i) \right) (\lambda - 1).
	\label{eq:Psi_def}
	\end{equation}
	We shall prove that $\Psi$ is an annihilating polynomial of the preconditioned matrix, i.e., $\Psi\left( \mathscr{P}_{\mathrm{VBP}}^{-1} \mathscr{A} \right) = 0$.
	
	Substituting the block form \eqref{eq:precond_block} into~\eqref{eq:Psi_def}, we obtain
	\begin{align}
	\Psi\left( \mathscr{P}_{\mathrm{VBP}}^{-1} \mathscr{A} \right)
	&= \left[ \prod_{i=1}^n \left( \mathscr{P}_{\mathrm{VBP}}^{-1} \mathscr{A} - \lambda_i I \right) \right] \left( \mathscr{P}_{\mathrm{VBP}}^{-1} \mathscr{A} - I \right) \nonumber \\
	&= \left[ \prod_{i=1}^n \begin{bmatrix} \Theta_1 - \lambda_i I & \Theta_2 \\ 0 & (1 - \lambda_i) I \end{bmatrix} \right] \begin{bmatrix} \Theta_1 - I & \Theta_2 \\ 0 & 0 \end{bmatrix}\\
	&=\begin{bmatrix} \prod_{i=1}^n (\Theta_1 - \lambda_i I)(\Theta_1-I) & \prod_{i=1}^n (\Theta_1 - \lambda_i I)\Theta_2 \\ 0 & 0 \end{bmatrix}
	 \label{eq:product_step}
	\end{align}
	
 According to the Cayley–Hamilton theorem, the matrix $\Theta_1$ satisfies its own characteristic equation, meaning $\prod_{i=1}^n (\Theta_1 - \lambda_i I) = 0$. Substituting this into \eqref{eq:product_step}, we obtain 
	\begin{equation}
	\Psi\left( \mathscr{P}_{\mathrm{VBP}}^{-1} \mathscr{A} \right)
	= \begin{bmatrix} 0 & 0 \\ 0 & 0 \end{bmatrix}.
	\end{equation}
	
	 Therefore, the degree of the minimal polynomial of the preconditioned matrix $\mathscr{P}_{\mathrm{VBP}}^{-1} \mathscr{A}$ is at most $n + 1$.
\end{proof}

\begin{remark}
	The theoretical results in Theorem \ref{thm:minimal_poly} determine the convergence behavior of a Krylov subspace method, such as GMRES \cite{Saadgmres}. Theorem \ref{thm:minimal_poly} shows that with the VBP preconditioner $\mathscr{P}_{\mathrm{VBP}}$, termination (in exact arithmetic) of the GMRES method will occur in at most $n + 1$ steps for any choice of the right hand side $\mathbf{d}$.
\end{remark}

\section{Parameter selection}\label{Sec4}
The overall efficiency of the $ \mathscr{P}_{\mathrm{VBP}} $ is highly sensitive to the selection of the two parameters $\alpha$ and $\gamma$. Therefore, to achieve optimal performance, it is crucial to establish an effective method for estimating these values within the VBP framework.

As indicated by \eqref{w1}, the deviation between the preconditioner $\mathscr{P}_{\mathrm{VBP}} $ and the  matrix $\mathscr{A}$ highlights the need for proper parameter tuning. Ideally, these parameters should be chosen to minimize this distance, ensuring that the preconditioner closely approximates the original coefficient matrix. Although several strategies for parameter selection have been explored in the literature (see, e.g., \cite{Bai2000, Benzi2016, Chen2015, Huang2014}), this study adopts the algebraic estimation technique proposed by Huang \cite{Huang2014} to determine the practical values of $\alpha$ and $\gamma$.
Based on Eq.~\eqref{eq:matrix_R}, we have $\mathscr{R}$ as follows
\begin{equation} \label{eq:matrix_R}
\mathscr{R}_{\mathrm{VBP}}  = 
\begin{bmatrix} 
\alpha I - \gamma W^2 & W - \gamma WT \\ 
0 & 0 
\end{bmatrix}.
\end{equation}
To determine the quasi-optimal parameters, we aim to minimize the overall distance between the preconditioner and the original matrix. Therefore, we introduce an objective function $f(\alpha, \gamma)$ which relies on the Frobenius norm of the matrix $\mathrm{R}$
\begin{equation} \label{eq:norm_F}
f(\alpha, \gamma) = \|\mathscr{R}_{\mathrm{VBP}}\|_F^2 = \operatorname{tr}(\mathscr{R}_{\mathrm{VBP}}^{T}\mathscr{R}_{\mathrm{VBP}}).
\end{equation}
By minimizing this objective function with respect to the involved parameters, we can analytically derive their quasi-optimal values. Consequently, we have
\begin{equation} \label{eq:trace_full}
\begin{aligned}
f(\alpha, \gamma)=\operatorname{tr}(\mathscr{R}_{\mathrm{VBP}}^{T}\mathscr{R}_{\mathrm{VBP}}) = &\ \alpha^2 n - 2\alpha\gamma \operatorname{tr}(W^2) + \gamma^2 \operatorname{tr}(W^4) \\
& + \operatorname{tr}(W^2) - 2\gamma \operatorname{tr}(W^2 T) + \gamma^2 \operatorname{tr}(W T^2 W).
\end{aligned}
\end{equation}
To obtain the quasi-optimal parameters in the VBP preconditioner, we regard the parameter $\gamma$ as a constant and analyze only the parameter $\alpha$. By minimizing the function $f(\alpha, \gamma)$, we get the quasi-optimal parameter $\alpha$ in the VBP preconditioner as follows

\begin{equation}\label{alphaVBP}
\alpha_{\text{VBP}}=\alpha_{\mathrm{\text{qopt}}} = \frac{\gamma \|W\|_F^2}{n}
\end{equation}

\section{Numerical Experiments}\label{Sec5}
In this section, we evaluate the numerical performance of the proposed $ \mathscr{P}_{\mathrm{VBP}} $ for~\eqref{orginal} through a series of comprehensive numerical experiments. All computations are carried out using \textsc{Matlab} (R2020a) on a Windows 10 system equipped with an Intel Core i5 processor running at 2.6 GHz and 8 GB of RAM.

In each experiment, the iterative process is initiated with a zero initial guess, i.e., $\mathbf{x}_0 = \mathbf{0}$. This approach ensures a consistent and reliable assessment of the convergence behavior and the overall effectiveness of the proposed preconditioner.

We compare the performance of the proposed preconditioner $\mathscr{P}_{\mathrm{VBP}}$ with three existing  preconditioners, namely $\mathscr{P}_{\mathrm{BS}}$~\cite{ZhangDai2017}, $\mathscr{P}_{\mathrm{VHSS}}$~\cite{ShenShi2018}, and $\mathscr{P}_{\mathrm{MBP}}$~\cite{BalaniHajarian2023}, whose formulations are given in \eqref{BS}, \eqref{VHSS}, and \eqref{MBP}, respectively.
 Proper parameter selection is essential for achieving fast convergence with the preconditioned FGMRES solver. 
 According to the parameter selection strategies proposed in~\cite{ShenShi2018}, the quasi-optimal parameter $\alpha$ for both the VHSS preconditioners is chosen as
\begin{equation}\label{alphaVHSS}
\alpha_{\text{VHSS}} = \sqrt{\lambda_{\min}(T) \, \lambda_{\max}(T)},
\end{equation}
where $\lambda_{\min}(T)$ and $\lambda_{\max}(T)$ denote the smallest and largest eigenvalues of the matrix $T$, respectively.

Following the approach introduced in~\cite{ZhangDai2017}, the parameter $\alpha_{\text{BS}}$ for the BS preconditioner is determined by
\begin{equation}\label{alphaBS}
\alpha_{\text{BS}} = \sqrt[4]{\frac{\text{tr}(T W^2 T)}{n}},
\end{equation}
where $\text{tr}(\cdot)$ denotes the trace operator and $n$ is the dimension of the matrices involved.

Furthermore, for the MBP preconditioner, the parameter $\alpha_{\text{MBP}}$ is set according to the formulation given in~\cite{BalaniHajarian2023} as
\begin{equation}\label{alphaMBP}
\alpha_{\text{MBP}} = \frac{\delta_{\max}^2 + \delta_{\min}^2 + 2}{2},
\end{equation}
where $\delta_{\max}$ and $\delta_{\min}$ are defined in~\cite{BalaniHajarian2023}.
 Since our proposed preconditioner is two-parameter based, we fix the parameter $\gamma = 1e-06$ throughout all numerical experiments. 

The flexible GMRES (FGMRES) method was applied to solve the preconditioned systems, allowing at most 1000 iterations and using a stopping tolerance of $10^{-6}$ based on the relative residual
\[
\text{RES}=\frac{\|\mathbf{d} - \mathcal{\tilde{A}}\mathbf{x}_k\|_2}{\|\mathbf{d}\|_2} \leq 10^{-6}.
\]
To measure solution quality, we record the relative error (ERR)
\[
\text{ERR} = \frac{\|\mathbf{x}_k - \mathbf{x}_*\|_2}{\|\mathbf{x}_*\|_2},
\]
where $\mathbf{x}_k$ is the computed solution after $k$ steps and $\mathbf{x}_*$ is the exact one.

The inner linear systems arising within the FGMRES framework are solved inexactly using the conjugate gradient (CG) method. For each inner iteration, the CG solver is employed with a relative tolerance of $10^{-2}$ and a maximum allowable number of iterations set to 500. To ensure statistical reliability, all reported CPU times (in seconds) and iteration counts are averaged over three independent runs.

\begin{example}[\cite{CaoRen2015,ZhangDai2017}]\label{EX1}
	As the first test problem, we consider the complex symmetric linear system of the form~\eqref{EQ1}, given by
	\begin{equation}\label{KW}
	\left[ \left( \mathbf{K} - (3 - \sqrt{3}) \omega^2 \mathbf{I} \right) + i \left( \mathbf{K} + (3 + \sqrt{3}) \tau^2 \mathbf{I} \right) \right]\mathbf{ x} =\mathbf{d},
	\end{equation}
	where $\tau$ and $\omega$ are positive parameters. The matrix $\mathbf{K} \in \mathbb{R}^{n \times n}$ represents the five-point centered difference approximation of the negative Laplacian operator subject to homogeneous Dirichlet boundary conditions on the unit square $[0,1] \times [0,1]$ with uniform mesh size $h = \frac{1}{m+1}$. Specifically, $\mathbf{K}$ admits the tensor-product form
	\begin{equation}
	\mathbf{K} = \mathbf{I} \otimes \mathbf{V}_m + \mathbf{V}_m \otimes \mathbf{I},
	\end{equation}
	with $\mathbf{V}_m = h^{-2} \, \text{tridiag}(-1, 2, -1) \in \mathbb{R}^{m \times m}$. Consequently, $\mathbf{K}$ is a block tridiagonal matrix of dimension $n = m^2$.
	
	In our numerical experiments, the right-hand side vector is chosen as $\mathbf{d} = (1 + i) \mathscr{A}*\textbf{ones}(2m^2, 1).$, and we set $\tau = 1$. Additionally, the linear system~\eqref{KW} is normalized by multiplying both sides by $h^2$. To facilitate the application of preconditioning techniques, we define the real matrices
	\begin{equation}
	W = h^2 \left( \mathbf{K} - (3 - \sqrt{3}) \omega^2 \mathbf{I} \right), \qquad
	T = h^2 \left( \mathbf{K} + (3 + \sqrt{3}) \tau^2 \mathbf{I} \right).
	\end{equation}
	
	From the spectral analysis presented in~\cite{LundBowers1992}, the eigenvalues of $W$ and $T$ are known to lie within the following intervals
	\begin{equation}
	\lambda(W) \in \left[ -\frac{(3 - \sqrt{3}) \omega^2}{(m+1)^2} + 4 \left( 1 - \cos \frac{\pi}{m+1} \right), -\frac{(3 - \sqrt{3}) \omega^2}{(m+1)^2} + 4 \left( 1 - \cos \frac{m \pi}{m+1} \right) \right],
	\end{equation}
	and
	\begin{equation}
	\lambda(T) \in \left[ \frac{(3 + \sqrt{3}) \tau^2}{(m+1)^2} + 4 \left( 1 - \cos \frac{\pi}{m+1} \right), \frac{(3 + \sqrt{3}) \tau^2}{(m+1)^2} + 4 \left( 1 - \cos \frac{m \pi}{m+1} \right) \right].
	\end{equation}
	
	From these eigenvalue bounds, it can be readily verified that the matrix $T$ is symmetric positive definite. Moreover, the matrix $W$ is symmetric indefinite provided that the parameter $\omega$ satisfies the following condition
	\begin{equation}
	2(m+1) \sqrt{\frac{1 - \cos \frac{\pi}{m+1}}{3 - \sqrt{3}}} < \omega < 2(m+1) \sqrt{\frac{1 - \cos \frac{m \pi}{m+1}}{3 - \sqrt{3}}}.
	\end{equation}
\end{example}

In our numerical experiments, the parameters are set to $\omega =10, 20$ and $\tau = 1$, which result in a symmetric indefinite matrix $W$. To evaluate the performance across different problem sizes, three grid levels with $m = 32$, $64$, and $128$ are considered. The quasi-optimal parameters for the preconditioners $\mathscr{P}_{\mathrm{BS}}$, $\mathscr{P}_{\mathrm{VHSS}}$, $\mathscr{P}_{\mathrm{MBP}}$, and $\mathscr{P}_{\mathrm{VBP}}$ are computed using the formulas given in \eqref{alphaBS}, \eqref{alphaVHSS}, \eqref{alphaMBP}, and \eqref{alphaVBP}, respectively. The corresponding parameter values are summarized in Tables~\ref{Tab1} and~\ref{Tab2}.

Figure~\ref{fig:eigenvalues} illustrates the eigenvalue distribution of the preconditioned matrices corresponding to the four preconditioners under investigation—namely $\mathscr{P}_{\text{BS}}$, $\mathscr{P}_{\text{VHSS}}$, $\mathscr{P}_{\text{MBP}}$, and the proposed preconditioner $\mathscr{P}_{\text{VBP}}$. This figure is generated for a grid of size $16 \times 16$ using the quasi-optimal parameters for each preconditioner. For the proposed $\mathscr{P}_{\text{VBP}}$ preconditioner, the parameter $\gamma$ is set to $0.282$. As can be observed, all four preconditioners yield a relatively favorable spectral distribution, indicating that each method effectively improves the conditioning of the original coefficient matrix. However, it is evident that the eigenvalues for the proposed preconditioner $\mathscr{P}_{\text{VBP}}$ are significantly more tightly clustered around the point $(1, 0)$ in the complex plane compared to the other methods.

Furthermore, the eigenvalues associated with $\mathscr{P}_{\text{VBP}}$ lie entirely within the positive real interval $[1, 1.19]$, which is fully consistent with the theoretical result established in Theorem~\ref{th3}.

 \begin{figure}[!htp]
	\centering
	\includegraphics[scale=0.90]{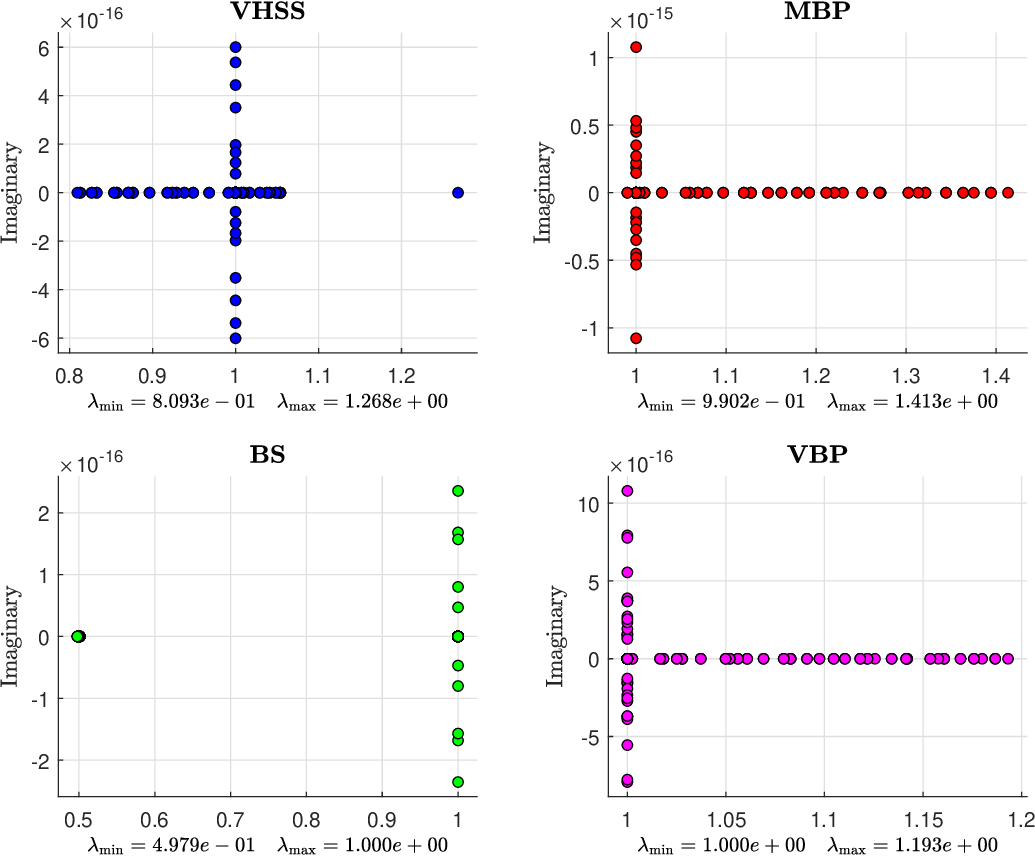}
	\caption{Eigenvalue distribution of the preconditioned matrices for the quasi-optimal parameters with $\tau = 1$ and $\omega = 10$ for Example~\ref{EX1}.}
	\label{fig:eigenvalues}
\end{figure}

 The numerical results presented in Tables~\ref{Tab3} and~\ref{Tab4} demonstrate the performance of FGMRES with four different preconditioners ($\mathscr{P}_{\text{BS}}$, $\mathscr{P}_{\text{VHSS}}$, $\mathscr{P}_{\text{MBP}}$, and $\mathscr{P}_{\text{VBP}}$) in terms of iteration counts, CPU times, and optimal parameter $\alpha$ for three grid sizes. These tables correspond to the parameter settings $(\tau = 1, \omega = 10)$ and $(\tau = 1, \omega = 20)$, respectively.
 
 The results clearly indicate that the proposed preconditioner $\mathscr{P}_{\text{VBP}}$ achieves the fastest convergence, requiring the fewest iterations and the shortest CPU time across all grid levels.

\begin{table}[!htbp]
	\centering
	\caption{Comparison of quasi-optimal parameter $\alpha$ values obtained by different preconditioning methods for Example~\ref{EX1} across various grid sizes with $ \tau=1 $ and $ \omega=10 $ .}
	\label{Tab1}
	\renewcommand{\arraystretch}{1.2}  
	\begin{tabular}{l|c|c|c|c}
		\toprule
		Method & $\alpha$ & $32\times 32$ & $64\times 64$ & $128\times 128$ \\
		\midrule
		$\mathscr{P}_{\text{BS}}$    & $\alpha_{\text{BS}}$   & 5.0245  & 5.0734  & 5.0892  \\
		$\mathscr{P}_{\text{VHSS}}$  & $\alpha_{\text{VHSS}}$ & 4.2350e-01 & 2.159e-01 & 1.0846e-01 \\
		$\mathscr{P}_{\text{MBP}}$   & $\alpha_{\text{MBP}}$  & 10.5836 & 10.5730 &10.5702 \\
		$\mathscr{P}_{\text{VBP}}$   & $\alpha_{\text{VBP}}$  & 1.8957e-5 & 1.9698e-05 & 1.9908e-05 \\
		\bottomrule
	\end{tabular}
\end{table}

\begin{table}[!htbp]
	\centering
	\caption{Comparison of optimal parameter $\alpha$ values obtained by different preconditioning methods for Example~\ref{EX1} across various grid sizes with $ \tau=1 $ and $ \omega=20 $.}
	\label{Tab2}
	\renewcommand{\arraystretch}{1.2} 
	\begin{tabular}{l|c|c|c|c}
		\toprule
		Method & $\alpha$ & $32\times 32$ & $64\times 64$ & $128\times 128$ \\
		\midrule
		$\mathscr{P}_{\text{BS}}$    & $\alpha_{\text{BS}}$   & 4.8735  & 5.0351  & 5.0796  \\
		$\mathscr{P}_{\text{VHSS}}$  & $\alpha_{\text{VHSS}}$ & 4.2350e-01 & 2.150e-01 & 1.0846e-01 \\
		$\mathscr{P}_{\text{MBP}}$   & $\alpha_{\text{MBP}}$  & 199.5352 & 199.4467 & 199.3979 \\
		$\mathscr{P}_{\text{VBP}}$   & $\alpha_{\text{VBP}}$  & 1.6360e-5 & 1.8999e-05 & 1.9726e-05 \\
		\bottomrule
	\end{tabular}
\end{table}

\begin{figure}[!htp]
	\centering
	\includegraphics[scale=0.95]{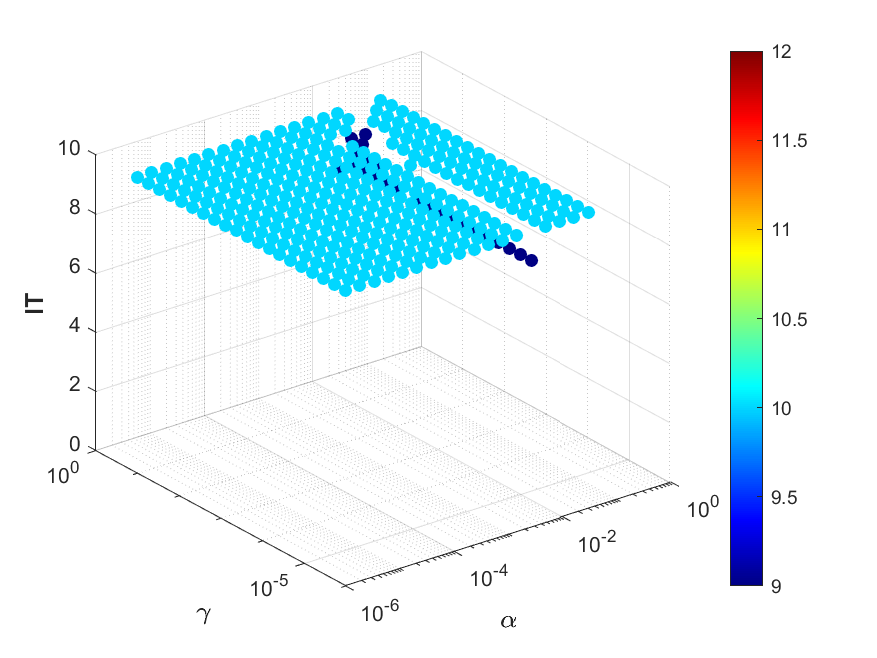}
	\caption{Iteration counts of the VBP preconditioned FGMRES method versus parameters $\alpha$ and $\gamma$ for Example~\ref{EX1} on a $32 \times 32$ grid with $ \tau=1 $ and $ \omega=10 $.}
	\label{fig:iter}
\end{figure}

\begin{table}[htbp]
	\centering
	\caption{Comparison of FGMRES performance using different preconditioning strategies with quasi-optimal parameter $\alpha$ for Example~\ref{EX1} with $ \tau=1 $ and $ \omega=10 $.}
	\label{Tab3}
	\setlength{\tabcolsep}{8pt}
	\renewcommand{\arraystretch}{1.4}
	\begin{tabular}{lccccccccc}
		\toprule
		\multirow{2}{*}{Grids} &
		\multicolumn{3}{c}{$32\times 32$} & 
		\multicolumn{3}{c}{$64\times 64$} & 
		\multicolumn{3}{c}{$128\times 128$} \\
		\cmidrule(lr){2-4} \cmidrule(lr){5-7} \cmidrule(lr){8-10}
		& IT & CPU & RES & IT & CPU & RES & IT & CPU & RES \\
		\midrule
		$\mathscr{P}_{\mathrm{BS}}$   & 29 & 0.26 & 8.3e-07 & 54 & 3.09 & 9.2e-07 & 103 & 38.43 & 8.8e-07 \\
		$\mathscr{P}_{\mathrm{VHSS}}$ & 15 & 0.19 & 6.1e-07 & 21 & 3.11 & 5.9e-07 & 41  & 32.66 & 7.4e-07 \\
		$\mathscr{P}_{\mathrm{MBP}}$  & 14 & 0.10 & 2.1e-07 & 13 & 0.46 & 7.0e-07 & 12  & 2.39  & 6.3e-07 \\
		$\mathscr{P}_{\mathrm{VBP}}$  & 10 & 0.05 & 2.4e-07 & 9  & 0.29 & 9.8e-07 & 9   & 1.41  & 5.5e-07 \\
		\bottomrule
	\end{tabular}
\end{table}
\begin{table}[htbp]
	\centering
	\caption{Comparison of FGMRES performance using different preconditioning strategies with quasi-optimal parameter $\alpha$ for Example~\ref{EX1} with $ \tau=1 $ and $ \omega=20 $.}
	\label{Tab4}
	\setlength{\tabcolsep}{8 pt}
	\renewcommand{\arraystretch}{1.4}
	\begin{tabular}{lccccccccc}
		\toprule
		\multirow{2}{*}{Grids} &
		\multicolumn{3}{c}{$32\times 32$} & 
		\multicolumn{3}{c}{$64\times 64$} & 
		\multicolumn{3}{c}{$128\times 128$} \\
		\cmidrule(lr){2-4} \cmidrule(lr){5-7} \cmidrule(lr){8-10}
		& IT & CPU & RES & IT & CPU & RES & IT & CPU & RES \\
		\midrule
		$\mathscr{P}_{\mathrm{BS}}$   & 22 & 0.16 & 4.8e-07 & 39 & 1.93 & 9.0e-07 & 72 & 25.18 & 9.0e-07 \\
		$\mathscr{P}_{\mathrm{VHSS}}$ & 16 & 0.12 & 7.9e-07 & 22 & 1.46 & 6.5e-07 & 29 & 21.33 & 9.2e-07 \\
		$\mathscr{P}_{\mathrm{MBP}}$  & 27 & 0.17 & 5.7e-07 & 25 & 0.96 & 9.9e-07 & 24 & 5.26  & 7.3e-07 \\
		$\mathscr{P}_{\mathrm{VBP}}$  & 14 & 0.06 & 7.9e-07 & 14 & 0.42 & 2.3e-07 & 13 & 2.00  & 9.0e-07 \\
		\bottomrule
	\end{tabular}
\end{table}
Figure~\ref{fig:iter} displays the iteration counts of the VBP-preconditioned FGMRES 
method versus parameters $ \alpha $ and $ \gamma $ for Example~\ref{EX1} on a $32 \times 32$ grid.
The figure clearly demonstrates the insensitivity of the method to parameter variations, as the iteration counts stay low across a wide parameter domain.

 \begin{example}[\cite{ZhangDai2017}]\label{EX2}
 	In this example, we consider the complex symmetric indefinite linear system of the form (3), given by
 	\begin{equation}\label{KCV}
 	\left[(-\omega^2 M + K) + i(\omega C_V + C_H)\right] \mathbf{x} = \mathbf{d},
 	\end{equation}
 	Here, $M$ and $K$ are the inertia and stiffness matrices, respectively; $C_V$ and $C_H$ are the viscous and hysteretic damping matrices, respectively; and $\omega$ is the driving circular frequency. In our numerical experiments, we set $M = I$, $C_V = 10I$, and $C_H = \tau K$ with $\tau = 10$. The matrix $K$ is constructed in the same manner as described in Example~\ref{EX1}. Furthermore, we choose the right-hand side vector $\mathbf{d}$ is selected as $\mathbf{d} = (1 + i)\mathscr{A}\textbf{Ones}(2m^2,1)$. Finally, the complex linear system~\eqref{KCV} is normalized by multiplying both sides by $h^2$, consistent with the previous example.
 \end{example}

In the present numerical study, we adopt the parameter values $\omega =10, 20$ and $\tau = 10$, leading to a symmetric indefinite coefficient matrix $W$. To investigate the scalability of the proposed approach, we perform simulations on three distinct grid resolutions corresponding to $m = 32$, $64$, and $128$. For each of the preconditioners $\mathscr{P}_{\mathrm{BS}}$, $\mathscr{P}_{\mathrm{VHSS}}$, $\mathscr{P}_{\mathrm{MBP}}$, and $\mathscr{P}_{\mathrm{VBP}}$, the quasi-optimal parameters are determined via the respective formulas in \eqref{alphaVHSS},\eqref{alphaBS},  \eqref{alphaMBP}, and \eqref{alphaVBP}. The computed parameter values are then reported in Tables~\ref{Tab5} and~\ref{Tab6}.

Figure~\ref{fig1:eigenvalues} displays the spectral distribution of the preconditioned matrices associated with the four preconditioners considered in this study, namely $\mathscr{P}_{\text{BS}}$, $\mathscr{P}_{\text{VHSS}}$, $\mathscr{P}_{\text{MBP}}$, and the proposed $\mathscr{P}_{\text{VBP}}$. The results are obtained on a $16 \times 16$ grid with the quasi-optimal parameters selected for each preconditioner. In the case of the proposed $\mathscr{P}_{\text{VBP}}$ preconditioner, the parameter $\gamma$ is fixed at $0.14$. It is observed that the $\mathscr{P}_{\text{BS}}$, $\mathscr{P}_{\text{VHSS}}$, and $\mathscr{P}_{\text{VBP}}$ preconditioners produce a reasonably well-behaved spectral distribution, suggesting that each is capable of improving the conditioning of the original system. Nevertheless, the eigenvalues of the proposed preconditioner $\mathscr{P}_{\text{VBP}}$ exhibit a noticeably denser clustering around $(1, 0)$ in the complex plane relative to the other methods, indicating a superior spectral property.

 \begin{figure}[!htp]
	\centering
	\includegraphics[scale=0.90]{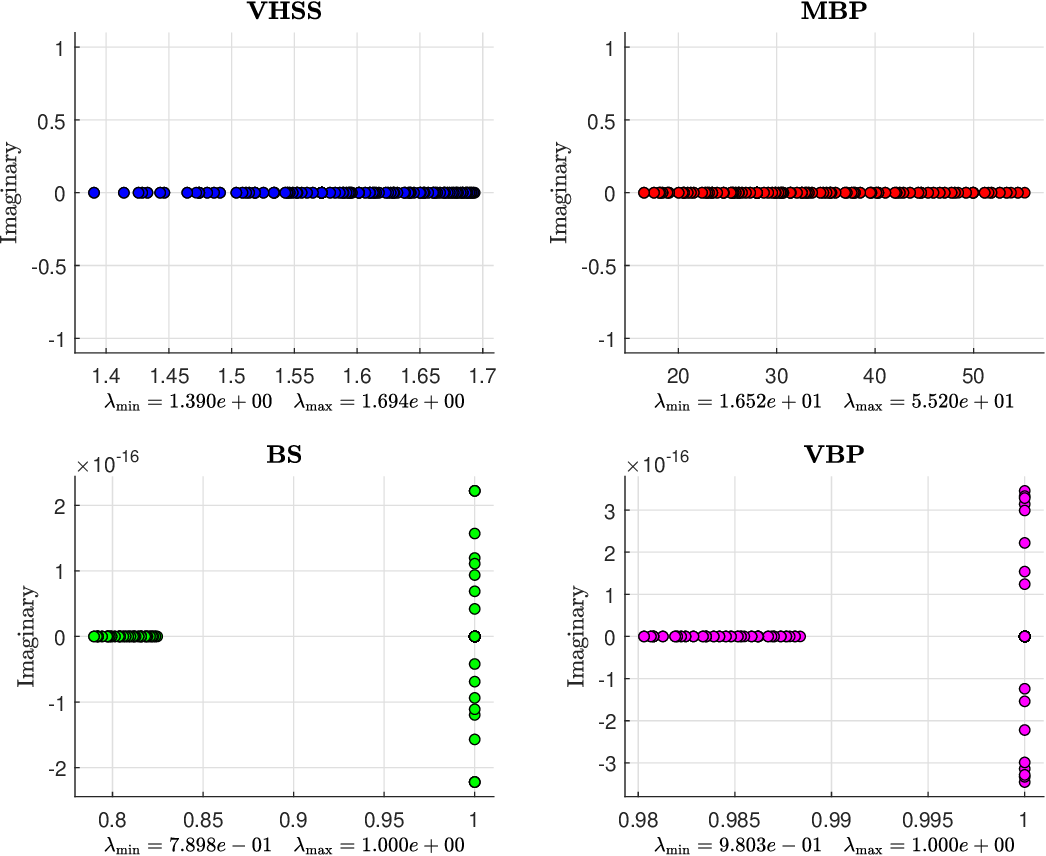}
\caption{Eigenvalue distribution of the preconditioned matrices for the quasi-optimal parameters with $\tau = 10$ and $\omega = 20$ for Example~\ref{EX2}.}
	\label{fig1:eigenvalues}
\end{figure}

\begin{table}[!htbp]
	\centering
	\caption{Comparison of quasi-optimal parameter $\alpha$ values obtained by different preconditioning methods for Example~\ref{EX2} across various grid sizes with $ \tau=10 $ and $ \omega=10 $ .}
		\renewcommand{\arraystretch}{1.5} 
	\label{Tab5}
	\renewcommand{\arraystretch}{1.2}  
	\begin{tabular}{l|c|c|c|c}
		\toprule
		Method & $\alpha$ & $32\times 32$ & $64\times 64$ & $128\times 128$ \\
		\midrule
		$ \mathscr{P}_{\text{BS}} $ & 	$\alpha_{\text{BS}}$&1.5929  &$ 1.6054 $&$ 1.6096 $\\
		$ \mathscr{P}_{\text{VHSS}} $ & 	$\alpha_{\text{VHSS}}$&4.6703 & $ 2.3725 $ &$ 1.1956 $\\
		$ \mathscr{P}_{\text{MBP}} $& $\alpha_{\text{MBP}}$&1.0365	& $ 1.0364 $ &$  1.0364 $\\
		$ \mathscr{P}_{\text{VBP}} $ & $\alpha_{\text{VBP}}$ &1.9149e-05& 1.9749e-05  & 1.9940e-05\\
		\toprule
	\end{tabular}
\end{table}

\begin{figure}[!htp]
	\centering
	\includegraphics[scale=0.95]{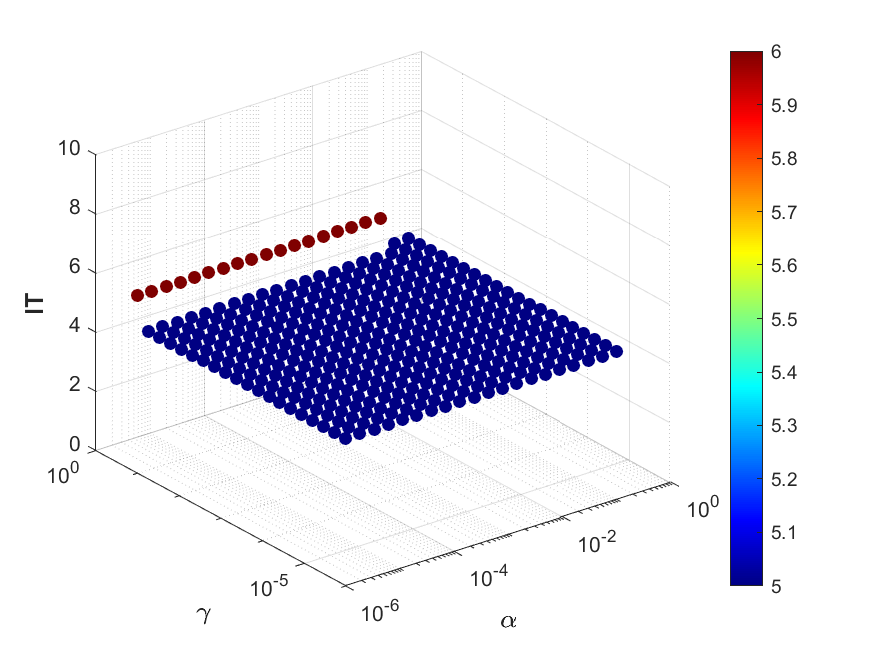}
	\caption{Iteration counts of the VBP preconditioned FGMRES method versus parameters $\alpha$ and $\gamma$ for Example~\ref{EX2} on a $32 \times 32$ grid with $ \tau=10 $ and $ \omega=20 $.}
	\label{fig:iter:Exam2}
\end{figure}

\begin{table}[!htbp]
	\centering
	\caption{Comparison of quasi-optimal parameter $\alpha$ values obtained by different preconditioning methods for Example~\ref{EX2} across various grid sizes with $ \tau=10 $ and $ \omega=20 $ .}
	\renewcommand{\arraystretch}{1.2} 
	\label{Tab6}
	\begin{tabular}{l|c|c|c|c}
		\toprule
		Method & $\alpha$ & $32\times 32$ & $64\times 64$ & $128\times 128$ \\
		\midrule
		$ \mathscr{P}_{\text{BS}} $ & 	$\alpha_{\text{BS}}$& 1.5565e+01   & 1.5962e+01 & 1.6673e+01 \\
		$ \mathscr{P}_{\text{VHSS}} $ & 	$\alpha_{\text{VHSS}}$&$ 5.4021 $ & $ 2.7430 $ &$ 1.3822 $\\
		$ \mathscr{P}_{\text{MBP}} $& $\alpha_{\text{MBP}}$&$ 1.4582 $	& $ 1.1.4579 $ &$  1.4578 $\\
		$ \mathscr{P}_{\text{VBP}} $ & $\alpha_{\text{VBP}}$ & 1.7071e-05 & 1.9189e-05  & 1.9770e-5\\
		\toprule
			\end{tabular}
\end{table}
Figure~\ref{fig:iter:Exam2} shows the IT counts of the VBP-preconditioned FGMRES versus parameters \(\alpha\) and \(\gamma\) for Example~\ref{EX2} on a \(32\times32\) grid.
This figure highlights the parameter insensitivity of the 
VBP preconditioner, as the iteration counts remain consistently low across 
a broad range of parameter values.

Tables~\ref{Tab7} and~\ref{Tab8} summarize the numerical performance of FGMRES when combined with four distinct preconditioners—namely $\mathscr{P}_{\text{BS}}$, $\mathscr{P}_{\text{VHSS}}$, $\mathscr{P}_{\text{MBP}}$, and the proposed $\mathscr{P}_{\text{VBP}}$—across three different grid sizes for Example~\ref{EX2}. The reported metrics include the number of iterations, the elapsed CPU time, and the optimal choice of the parameter $\alpha$. These two tables correspond to the parameter pairs $(\tau = 10, \omega = 10)$ and $(\tau = 10, \omega = 20)$, respectively.

Overall, the proposed $\mathscr{P}_{\text{VBP}}$ preconditioner consistently outperforms the other methods, yielding the lowest iteration counts and the least computational time for all tested grid resolutions.

\begin{table}[htbp]
	\centering
	\caption{Comparison of FGMRES performance using different preconditioning strategies with quasi-optimal parameter $\alpha$ for Example~\ref{EX2} with $ \tau=10 $ and $ \omega=10 $.}
	\label{Tab7}
	\setlength{\tabcolsep}{8pt}
	\renewcommand{\arraystretch}{1.4}
	\begin{tabular}{lccccccccc}
		\toprule
		\multirow{2}{*}{Grids} &
		\multicolumn{3}{c}{$32\times 32$} & 
		\multicolumn{3}{c}{$64\times 64$} & 
		\multicolumn{3}{c}{$128\times 128$} \\
		\cmidrule(lr){2-4} \cmidrule(lr){5-7} \cmidrule(lr){8-10}
		& IT & CPU & RES & IT & CPU & RES & IT & CPU & RES \\
		\midrule
		$\mathscr{P}_{\mathrm{BS}}$   & 24 & 0.14 & 7.6e-07 & 45 & 1.22 & 5.5e-07 & 85 & 20.22 & 7.6e-07 \\
		$\mathscr{P}_{\mathrm{VHSS}}$ & 14 & 0.07 & 1.6e-07 & 18 & 0.71 & 2.6e-07 & 24 & 5.45  & 2.8e-07 \\
		$\mathscr{P}_{\mathrm{MBP}}$  & 6  & 0.03 & 5.4e-07 & 7  & 0.21 & 4.2e-07 & 7  & 1.22  & 5.4e-07 \\
		$\mathscr{P}_{\mathrm{VBP}}$  & 5  & 0.02 & 3.1e-07 & 5  & 0.18 & 3.4e-08 & 4  & 0.82  & 6.4e-07 \\
		\bottomrule
	\end{tabular}
\end{table}

%
\begin{table}[htbp]
	\centering
	\caption{Comparison of FGMRES performance using different preconditioning strategies with quasi-optimal parameter $\alpha$ for Example~\ref{EX2} with $ \tau=10 $ and $ \omega=20 $.}
	\label{Tab8}
	\setlength{\tabcolsep}{8pt}
	\renewcommand{\arraystretch}{1.4}
	\begin{tabular}{lccccccccc}
		\toprule
		\multirow{2}{*}{Grids} &
		\multicolumn{3}{c}{$32\times 32$} & 
		\multicolumn{3}{c}{$64\times 64$} & 
		\multicolumn{3}{c}{$128\times 128$} \\
		\cmidrule(lr){2-4} \cmidrule(lr){5-7} \cmidrule(lr){8-10}
		& IT & CPU & RES & IT & CPU & RES & IT & CPU & RES \\
		\midrule
		$\mathscr{P}_{\mathrm{BS}}$   & 23 & 0.15 & 7.1e-07 & 42 & 1.44 & 9.0e-07 & 80 & 12.74 & 9.8e-07 \\
		$\mathscr{P}_{\mathrm{VHSS}}$ & 14 & 0.07 & 3.1e-07 & 18 & 0.69 & 4.6e-07 & 24 & 4.74  & 4.1e-07 \\
		$\mathscr{P}_{\mathrm{MBP}}$  & 7  & 0.04 & 2.8e-07 & 7  & 0.21 & 5.3e-07 & 7  & 1.10  & 4.7e-07 \\
		$\mathscr{P}_{\mathrm{VBP}}$  & 5  & 0.02 & 4.4e-07 & 5  & 0.16 & 7.0e-07 & 5  & 0.87  & 4.2e-07 \\
		\bottomrule
	\end{tabular}
\end{table}

\section{Conclusions}\label{Sec6}
In this paper, we have proposed a variant of the block preconditioner (VBP) based on a matrix splitting of the coefficient matrix. The convergence of the corresponding stationary iterative method was analyzed, and the spectral properties of the preconditioned matrix were investigated in detail. Numerical experiments confirm that the proposed preconditioner outperforms the other preconditioners considered in this study in terms of computational efficiency and robustness.

\section*{Conflicts of interest}
This work does not have any conflicts of interest.

\end{document}